\documentclass{article}
\usepackage[utf8]{inputenc}

\title{Complementary Schur Asymptotics for Partitions}
\author{Jaroslav Han\v cl Jr.}

\usepackage{graphicx}
\usepackage{amsmath}
\usepackage{amssymb}
\usepackage{amsthm}

\newtheorem{theorem}{Theorem}
\newtheorem{lemma}{Lemma}
\newtheorem{hypothesis}{Conjecture}

\begin{document}

\maketitle

\begin{abstract}
We deduce from the strong form of the Hardy--Ramanujan asymptotics for the partition function $p(n)$ an asymptotics for $p_{-S}(n)$, the number of partitions of $n$ that do not use parts from a finite set $S$ of positive integers. We apply this to construct highly oscillating partition ideals.

\end{abstract}

\section{Introduction}
Let $n \in \mathbb{N}$. We say that $\lambda = (\lambda_1, \dots, \lambda_k)$, where $\lambda_1 \geq \dots \geq \lambda_k > 0$ are positive integers, is a \emph{partition} of $n$ if $\sum_{i=1}^k \lambda_i = n$. We say that $k$ is the \emph{length} of $\lambda$ and $\lambda_1, \dots, \lambda_k$ are the parts of the partition $\lambda$.

Let $p(n)$ be the number of partitions of $n$. The famous result of Hardy and Ramanujan \cite{HR18} gives the asymptotics of $p(n)$ as
\begin{align} \label{eq:p(n)}
    p(n) \sim \frac{e^{C\sqrt{n}}}{4n\sqrt{3}}\;\mbox{ where $C = \pi \sqrt{2/3} \approx 2.565$}\;.
\end{align}
As usual, $f(n)\sim g(n)$ means 
that $\lim_{n\to\infty}\frac{f(n)}{g(n)}=1$. For more information on partitions refer 
to the book \cite{andr} of G. Andrews. In \cite{joha} F. Johansson presents an 
almost optimum numerical algorithm for evaluating $p(n)$.  

One can be also interested in the number $p_S(n)$ of partitions of $n$ with parts from a given (finite or infinite) set $S \subset \mathbb{N}$. Such partitions are often called ``restricted partitions''. I. Schur \cite{Schur26} proved that for finite $S$ with cardinality $|S|=t$ (and such that the $\mathrm{gcd}(S)=1$) one has the asymptotics
$$
    p_S(n) \sim \frac{n^{t-1}}{(t-1)! \prod_{s \in S}s} = \frac{1}{n(t-1)!} \prod_{s \in S} \frac{n}{s}.
$$
Note that adding a new element $s$ to $S$ increases $p_S(n)$ by the factor $\frac{n}{s|S|}$.

Motivated by these two partition asymptotics, we obtain what we call a complementary Schur asymptotic. Nicolas and Sárközy \cite{NiSa00} found an asymptotics of the number of partitions using only parts $\geq m$ for a wide range of parameter $m$. We extend this result. For a finite set of integers $S \subset \mathbb{N}$, $|S|=t$, let $p_{-S}(n)$ be the number of partitions of $n$ not using any part from $S$. Our main result is that 
$$
    p_{-S}(n) \sim p(n) \left( \frac{C}{2 \sqrt{n}} \right)^t\prod_{s\in S}s = p(n) \prod_{s \in S} \frac{Cs}{2\sqrt{n}}\;.
$$
Now each element $s$ in $S$ decreases $p(n)$ by the factor $\frac{Cs}{2\sqrt{n}}$. We apply this to construct highly oscillating partition ideals. This theorem is an extension of result \cite{NiSa00} of Nicolas and Sárközy who considered the case $S = [m]$. 

The paper is organized as follows. In Section 2 we state the strong Hardy--Ramanujan asymptotics of $p(n)$ and restate our main result. Section 3 presents some useful lemmas. Section 4 is devoted to the proof of the main theorem. In the last Section 5 we  present an application of the main result to oscillations of growth functions of partition ideals.

\section{Two asymptotics for partitions}
Let $n\in\mathbb{N}$, 
$$
    \lambda_n := \sqrt{n-\frac{1}{24}}, \quad C := \pi \sqrt{\frac23} \quad \text{and} \quad D\in(C/2,C)
$$
be a constant.

As we mentioned earlier, the classical asymptotics for $p(n)$ proved by Hardy and Ramanujan \cite{HR18} is $p(n)\sim e^{C\sqrt{n}}/4n\sqrt{3}$. In \cite[formula (1.55)]{HR18} they gave much stronger asymptotics which we give in the next theorem. They state it as a value of a derivative which we compute explicitly. 

\begin{theorem}[Hardy--Ramanujan] \label{eq:BestAsympt}
For $n=1,2,\dots$, the partition function $p(n)$ satisfies
\begin{align*}
    p(n) = \frac{e^{C\lambda_n}}{4 \pi \sqrt{2} \lambda_n^2} \left( C - \frac{1}{\lambda_n} \right) + O \left( e^{D\sqrt{n}} \right).
\end{align*}
\end{theorem}
Note that now the error is only about a square root of the main term. This resembles strong asymptotic relations for coefficients of power series with unique dominant singularity (see P. Flajolet and R. Sedgewick \cite[Chapter V]{flaj_sedg} or think of the Fibonacci numbers). We use this result with small error term to deduce our asymptotic relation for $p_{-S}(n)$. Our main theorem says:

\begin{theorem} \label{thm:main}
Let $S=\{s_1,s_2,\dots,s_t\} \subset \mathbb{N}$ be a finite set of integers with $|S| = t$. Then the number $p_{-S}(n)$ of partitions of $n$
with parts in $\mathbb{N}\backslash S$ satisfies
$$
    p_{-S}(n) \sim p(n) \left( \frac{C}{2 \sqrt{n}} \right)^t \prod_{s \in S} = p(n) \prod_{s \in S} \frac{Cs}{2\sqrt{n}}.
$$
\end{theorem}

We base our proof on manipulating the strong asymptotics of $p(n)$. First we estimate $p(n-s)$ for a fixed $s \in \mathbb{N}$, and then express $p_{-S}(n)$ as a sum of values $p(n-s)$ for various numbers $s$.

\section{Auxiliary results}

Let $S \subset \mathbb{N}$ be a finite set with $|S| = t$ and let $s \in [t]$. First we determine the asymptotics of $p(n-s)$ and then we prove an algebraic identity needed later in the proof of the main result.

We set
$$
    q(n) := \frac{e^{C \lambda_n}}{\lambda_n^2} \left( C - \frac{1}{\lambda_n} \right).
$$

\begin{lemma} \label{thm:1step}
Let $t\in\mathbb{N}$. Then for all $n,s \in \mathbb{N}$ with $n>s$ we have
\begin{align} \label{eq:mainterm}
     p(n-s) = \frac{e^{C \sqrt{n}}}{4 \pi n \sqrt{2}} \sum_{z=0}^{t} g(z, s) n^{-z/2} + O \left( e^{C\sqrt{n}} n^{-\frac{t+3}{2}}\right)
\end{align}
where for $z \in \{ 0, 1, \dots, t\}$ we denote by $g(z, s)$ a real polynomial in $s$ with degree $z$ and leading term
\begin{align} \label{eq:lead}
    g(z, s) = \frac{(-1)^z C^{z+1}}{2^z z!} s^{z} + h(z,s)
\end{align}
where $h(z,s)$ is a real polynomial in $s$ with degree at most $z-1$.
\end{lemma}

\begin{proof}
Let $t\in\mathbb{N}$.
We expand the main term in Theorem \ref{eq:BestAsympt} in powers of $n$. Since $e^x=\sum_{i\ge0}\frac{x^i}{i!}=\sum_{i=0}^l\frac{x^i}{i!}+O(x^{l+1})$ for $|x|<c<1$ and 
($c_i$ are some real constants, not necessarily always the same)
\begin{eqnarray*}
e^{C\lambda_n}&=&e^{C\sqrt{n}(1-1/24n)^{1/2}}=e^{C\sqrt{n}+c_1n^{-1/2}+c_2n^{-2/2}+\dots+c_l n^{-l/2}+O(n^{-(l+1)/2})}\\
\frac{1}{\lambda_n^2}&=&\frac{1}{n(1-1/24n)}=c_1n^{-1}+\dots+c_l n^{-l}+O(n^{-l-1})\\
C-\frac{1}{\lambda_n}&=&C-\frac{1}{\sqrt{n}(1-1/24n)^{1/2}}\\
&=&C+c_1n^{-1/2}+c_2n^{-2/2}+\dots+c_l n^{-l/2}+O(n^{-(l+1)/2})
\end{eqnarray*}
we get, for every $n>n_0(t)$ and some coefficients $a_k$,
\begin{align} \label{eq:star}
    q(n) = \frac{e^{C \lambda_n}}{\lambda_n^2} \left( C - \frac{1}{\lambda_n} \right) = e^{C\sqrt{n}} \left( \sum_{k=0}^{t} a_k n^{-k/2-1} + O(n^{-\frac{t+3}{2}}) \right)
\end{align}
(note that $a_0=C$). For all integers $n>s>0$ we have, expanding again $(n-s)^{-k/2-1}=n^{-k/2-1}(1-s/n)^{-k/2-1}=\dots$ by I. Newton's binomial theorem, that
\begin{align} 
    &\sum_{k=0}^{t} a_k (n-s)^{-k/2-1}\nonumber\\ 
    &= \sum_{k=0}^{t} a_k \left( \sum_{l=0}^{t} \binom{-1-k/2}{l} (-s)^l n^{-1-k/2-l}+O(n^{-k/2-t-2}) \right) \nonumber \\ 
    &= \sum_{w=0}^{2t} f(w, s) n^{-w/2-1} + O(n^{-t-2})  \label{eq:prvni}
\end{align}
where 
\begin{align} \label{eq:leading}
    f(w ,s) = \sum_{k,l\ge0,\,k+2l=w} (-1)^l a_k \binom{-1-k/2}{l} s^l
\end{align}
is a real polynomial in $s$ with $\deg_s f(w, s) = \lfloor w/2 \rfloor$. 

Next we combine expansions of the square root and of the exponential function, as above, and 
denote $x=n^{-1/2}$. For all integers $n>s>0$, where $s$ is fixed, we get 
\begin{eqnarray}  
    e^{C\sqrt{n-s}}&=&e^{C\sqrt{n}(1-s/n)^{1/2}}\nonumber\\
    &=&e^{C\sqrt{n}(1+\binom{1/2}{1}(-s/n)^1+\dots+\binom{1/2}{t}(-s/n)^t+O((s/n)^{t+1})}\nonumber\\
    &=&e^{C\sqrt{n}}e^{(-C/2)s x+c_2s^2x^3+c_3s^3x^5+\dots+c_ts^t x^{2t-1}+O(s^{t+1}x^{2t+1})}\nonumber\\
    &=& e^{C \sqrt{n}}\left(\sum_{i=0}^{t} d(i, s) n^{-i/2} +O(n^{-\frac{t+1}{2}})\right)\label{eq:exponential}
\end{eqnarray}
where $d(i, s)$ is a polynomial in $s$ with $\deg_s d(i, s) = i$ and 
\begin{align} \label{eq:twostars}
    d(i, s) = \frac{(-1)^i C^i}{2^i i!} s^{i} + c_{i-1}s^{i-1}+\dots
\end{align}
($c_{i-1}s^{i-1}+\dots$ are the remaining terms with degree less than $i$).

Combining results (\ref{eq:star}), (\ref{eq:prvni}) and (\ref{eq:exponential}) we have
\begin{align*}
   q(n-s) 
   =& e^{C \sqrt{n}}\left( \sum_{i=0}^{t} d(i, s) n^{-i/2} + O({n^{-\frac{t+1}{2}}}) \right) \left( \sum_{w=0}^{2t} f(w, s) n^{-\frac{w+2}{2}} + O({n^{-t-2}}) \right) \\
   =& \frac{e^{C \sqrt{n}}}{n} \sum_{z=0}^{t} g(z, s) n^{-z/2} + O \left( e^{C\sqrt{n}} n^{-\frac{t+3}{2}} \right),
\end{align*}
where 
$$
    g(z, s) = \sum_{i,w\ge0,\,i+w=z} d(i, s) f(w, s)
$$
is a real polynomial in $s$ with $\deg_s g(z, s) = z$. Indeed, $\max_{i,w\ge0,\,i+w=z}(i+\lfloor w/2\rfloor)=z$, attained uniquely for $i=z$ and $w=0$. Moreover, (\ref{eq:leading}) and (\ref{eq:twostars}) imply
\begin{align*}
    g(z, s) = \frac{(-1)^z C^{z+1}}{2^z z!} s^{z} + h(z, s)
\end{align*}
where $h(z, s)$ are the remaining terms with $\deg_s h(z, s) \leq z-1$. The exponent of $C$ increased by one since $a_0 = C$. Therefore for all integers $n>s>0$, 
\begin{align*}
   p(n-s) &= \frac{q(n-s)}{4\pi\sqrt{2}}+O\left(e^{D\sqrt{n}}\right)\\
   &= \frac{e^{C \sqrt{n}}}{4 \pi n \sqrt{2}} \sum_{z=0}^{t} g(z, s) n^{-z/2} + O \left( e^{C\sqrt{n}} n^{-\frac{t+3}{2}} \right).
\end{align*}
\end{proof}

We state and prove an identity needed in the main proof. Recall that for $t\in\mathbb{N}_0$, $[t]=\{1,2,\dots,t\}$.

\begin{lemma} \label{thm:algebra}
Let $t \geq z\ge0$ be integers and $s_1, s_2, \dots, s_t$ be variables. Then
$$
    \sum_{J \subset [t]} (-1)^{|J|} \left( \sum_{i \in J} s_i \right)^z = 
    \left\{ \begin{array}{ll}
      0                               & \mbox{for } z < t \\
	  (-1)^t t! \prod_{i=1}^t s_i     & \mbox{for } z = t.
	\end{array} \right. 
$$
The first case holds in fact more generally for any polynomial in $\sum_{i \in J} s_i$ with degree at most $t-1$.
\end{lemma}

\begin{proof}
Let $k \leq t$ be positive integers, $j_i$ with $1\le j_1<j_2<\dots<j_k\le t$ be some indices and $\alpha_1, \dots, \alpha_k$ be positive integers with $\sum \alpha_i = z$, so $k\le z$. We denote the polynomial on the left side as $f = f(s_1, \dots, s_t)$ and examine the coefficient  $[s_{j_1}^{\alpha_1} \cdots s_{j_k}^{\alpha_k}]f$. Clearly, only the sets $J$ containing $\{j_1,\dots,j_k\}$ contribute to it. Each such $J$ contributes $(-1)^{|J|} \binom{z}{\alpha_1, \dots, \alpha_k}$ to the coefficient. Summing over all $J$ containing $\{j_1,\dots,j_k\}$ we obtain
$$
    [s_{j_1}^{\alpha_1} \cdots s_{j_k}^{\alpha_k}]f = (-1)^k \binom{z}{\alpha_1, \dots, \alpha_k} \sum_{l=0}^{t-k} (-1)^l \binom{t-k}{l}
$$
(here $l=|J\backslash\{j_1,\dots,j_k\}|$). If $z<t$ then $k<t$ and the sum is $(1-1)^{t-k}=0$ by the binomial theorem. If $z=t$ then only $k=z=t$ yields non-zero contribution, for $\alpha_1 = \dots = \alpha_k = 1$, and we get the coefficient
$$
     [s_1 \cdots s_t]f = (-1)^t \binom{z}{1, \dots, 1} = (-1)^tz!=
     (-1)^t t!\;,
$$
which proves the theorem.
\end{proof}

\section{Proof of the Theorem \ref{thm:main}}

\begin{proof}[Proof of Theorem \ref{thm:main}]
Let $J\subset[t]$. Note that the partitions of $n$ using each part $s_j$, $j\in J$, at least once are in bijection with the partitions of $n-\sum_{j \in J} s_j$. Thus by the principle of inclusion and exclusion, 
$$
    p_{-S}(n) = \sum_{J \subset [t]} (-1)^{|J|} p\bigg(n - \sum_{j \in J} s_j\bigg)\;.
$$
By Lemma \ref{thm:1step} we have
\begin{align*}
    p_{-S}(n)
    =& \sum_{J \subset [t]} (-1)^{|J|} \left( \frac{e^{C \sqrt{n}}}{4 \pi n \sqrt{2}} \sum_{z=0}^{t} g\bigg(z, \sum_{j \in J} s_j\bigg) n^{-z/2} + O \left( e^{C\sqrt{n}} n^{-\frac{t+3}{2}} \right) \right) \\
    =& \frac{e^{C \sqrt{n}}}{4 \pi n \sqrt{2}} \sum_{z=0}^{t} \sum_{J \subset [t]} (-1)^{|J|} g\bigg(z, \sum_{j \in J} s_j\bigg) n^{-z/2} + O \left( e^{C\sqrt{n}} n^{-\frac{t+3}{2}} \right).
\end{align*}
We apply Lemma \ref{thm:algebra} to the polynomial $g(z, \sum_{j \in J} s_j)$ when $z \in [t-1]$---first we understand $s_j$ as variables and only at the end we substitute for them the numbers $s_j$---and obtain, by the first case,
$$
    \sum_{J \subset [t]} (-1)^{|J|} g(z, \sum_{j \in J} s_j) = 0\;. 
$$
Hence only term $z=t$ remains yielding
\begin{align*}
    p_{-S}(n) = \frac{e^{C \sqrt{n}}}{4 \pi n \sqrt{2}} \sum_{J \subset [t]} (-1)^{|J|} g\bigg(t, \sum_{j \in J} s_j\bigg) n^{-t/2} + O \left( e^{C\sqrt{n}} n^{-\frac{t+3}{2}} \right).
\end{align*}
Finally, we expand $g(t, \sum s_j)$ by equation (\ref{eq:lead}) and use Lemma \ref{thm:algebra}. By the first case of Lemma \ref{thm:algebra} the contributions of $h(z, \sum s_j)$ sum up to zero. By the second case  the contribution of the leading term in $g(z, \sum s_j)$ is
\begin{align*}
    \sum_{J \subset [t]} (-1)^{|J|} \frac{(-1)^t C^{t+1}}{2^t t!} \bigg(\sum_{j \in J} s_j\bigg)^{t} 
    = \frac{(-1)^t C^{t+1}}{2^t t!} (-1)^t t! \prod_{i=1}^t s_i 
    = \frac{C^{t+1}}{2^t} \prod_{i=1}^t s_i\;.
\end{align*}  
Thus
\begin{align*}
    p_{-S}(n) 
    =& \frac{e^{C \sqrt{n}}}{4 \pi n \sqrt{2}} \left[ \left(\frac{C^{t+1}}{2^t} \prod_{j=1}^t s_j \right)n^{-t/2} \right] + O \left( e^{C\sqrt{n}} n^{-\frac{t+3}{2}} \right),
\end{align*}
which in view that $C=\pi\sqrt{2/3}$ and $p(n)\sim e^{C\sqrt{n}}/4n\sqrt{3}$ gives the desired asymptotics
$$
     p_{-S}(n) \sim \frac{C e^{C \sqrt{n}}}{4 \pi n \sqrt{2}} \left(\frac{C}{2} \right)^t n^{-t/2} \prod_{j=1}^t s_j = p(n) \left( \frac{C}{2\sqrt{n}} \right)^t \prod_{j=1}^t s_j\;.
$$
\end{proof}

\section{Partition functions of ideals}

Let $X$ be a set of partitions and $\lambda, \gamma \in X$. We say that $\lambda$ is a \emph{subpartition} of $\gamma$ if no part from $\lambda$ has more occurrences in $\lambda$ then in $\gamma$. We denote this relation $\lambda < \gamma$. A set of partitions $X$ is a \emph{partition ideal}, if $\lambda < \gamma$ and $\gamma \in X$ always implies $\lambda \in X$. By $p(n, X)$ we denote the number of partitions of $n$ lying in $X$. Let $Z$ be a set of partitions such that no two partitions in $Z$ are comparable by $<$. We denote by $F_Z$ the set of all partitions that do not contain any element of $Z$. Clearly, $F_Z$ is a partition ideal. We call $Z$ a {\em basis} of the ideal $F_Z$. We denote its counting function by
$$
    p_{-Z}(n) = p(n, F_Z)\;.
$$
Recall that the notation $p_{-S}(n)$ and $p_{-Z}(n)$ where
$S$ is a set of positive integers and $Z$ is a set of partitions. Finally, two partitions are \emph{independent} if their parts are pairwise distinct.

We make use of the Cohen--Remmel theorem \cite{Cohen81, Remmel82} that gives sufficient condition for equality of counting functions of two partition ideals in terms of their bases. In the theorem we use the following notation. For a partition $\lambda$, $|\lambda|$ is the sum of all parts of $\lambda$, and for several partitions $\lambda^i$ their union is the partition $\lambda$ such that the multiplicity of any part equals to the maximum multiplicity attained over all $\lambda^i$.

\begin{theorem}[Cohen, 1981; Remmel, 1982]\label{cohe_remm}
Let $\Lambda = \{ \lambda^1, \lambda^2, \dots\}$ and $\Gamma = \{ \gamma^1 ,\gamma^2 , \dots \}$ be two finite or infinite sequences of partitions (of the same length), such that for every finite set $I \subset\mathbb{N}$, 
$$
\left|\bigcup_{i \in I} \lambda^i\right| = \left|\bigcup_{i \in I} \gamma^i\right|.
$$
Then $p_{-\Lambda}(n) = p_{-\Gamma}(n)$ for every positive integer $n$.
\end{theorem}

Now we state three applications of our Theorem~\ref{thm:main}. We are inspired by similar results of Han\v cl \cite{HanclDT}, but use a different approach.

\begin{theorem}
Let $F_Z$ be a partition ideal with basis $Z$, where $Z$ contains infinitely many pairwise independent partitions, and let $k$ be any positive integer. Then
$$
    p(n, F_Z) < K e^{C \sqrt{n}} n^{-k}
$$
for any sufficiently large $n$, where $K = K(k, Z)$ is a fixed constant (and $C=\pi\sqrt{2/3}$).
\end{theorem}

\begin{proof}
Let $k \in \mathbb{N}$ and $t = 2k-1$. Let $\lambda^1, \lambda^2, \dots, \lambda^t$  be mutually independent partitions from $Z$ such that $|\lambda^1| < |\lambda^2| < \cdots < |\lambda^t|$. By Theorem~\ref{cohe_remm}, applied to $\Lambda = \{\lambda^1, \lambda^2, \dots, \lambda^t\}$ and $\Gamma = \{(|\lambda^1|), (|\lambda^2|), \dots, (|\lambda^t|)\}$,
$$
    p(n, F_Z) \le p_{-\Lambda}(n) = p_{-\Gamma}(n).
$$
From Theorem \ref{thm:main} we have
$$
    p_{-\Gamma}(n) \sim  Ke^{C\sqrt{n}}n^{-1-t/2} < Ke^{C\sqrt{n}}n^{-k}
$$
where $K = K(k, Z)$ is a constant and the last inequality holds for any sufficiently large $n$.
\end{proof}

\begin{hypothesis}
Let $X = F_Z$ be a partition ideal with finite basis $Z$. Then the asymptotics for the counting function $p(n, X)$ is of the form
$$
    p(n, X) \sim K e^{C \sqrt{n}}n^{-1-k},
$$
where $K$ is a constant and $k = m/2$ for some $m \in \mathbb{N}$.
\end{hypothesis}


Let $\varepsilon \in(0,1)$. For the forthcoming theorem we set
$$
    f(n) = \left( 1 - \frac{\log^{1+\varepsilon} n}{\sqrt{n}} \right)^2.
$$
Thus $f(n)$ goes to $1$ as fast as $n^{-1/2} \log^{1+\varepsilon} n$ goes to $0$.

\begin{theorem}
Let $\varepsilon > 0$ and $f(n)$ be as above. Then there is a partition ideal $X$ such that both
\begin{enumerate}
    \item $p(n, X) = 0$
    \item $p(n, X) > p(nf(n))$
\end{enumerate}
holds for infinitely many positive integers $n$.  
\end{theorem}

\begin{proof}
We define the sequences $(s_i)_{i=1}^{\infty}$ and $(t_i)_{i=1}^{\infty}$ of positive integers such that $s_1 = 2$,
$$
    s_{i+1} = t_i^3 + 2
$$
and, given $s_i$, we set 
\begin{align} \label{eq:guessbound}
    t_i = \max \left\{ s_i, \exp \left( \left( \frac{3s_i+10}{2C}\right)^{1/\varepsilon} \right), 2n_0 \right\} \;,
\end{align}
where $n_0=n_0(s_i)$ is such that for any $n \geq n_0$ we have $f(n)\in(\frac{1}{2},1)$ and both 
$$
    2\cdot\frac{e^{C \sqrt{n}}}{4n \sqrt{3}}>p(n) > \frac12 \cdot \frac{e^{C \sqrt{n}}}{4n \sqrt{3}} \quad \text{ and } \quad p_{-[s_i]}(n) > \frac12 p(n) \prod_{s=1}^{s_i} \frac{Cs}{2\sqrt{n}}\;.
$$
Existence of $n_0$ follows from the Hardy--Ramanujan asymptotics (\ref{eq:p(n)}) and Theorem \ref{thm:main}.

Let $I_i = [s_i, t_i] \cap \mathbb{N}$. Let $X$ be the partition ideal consisting of the partitions that use parts from any of the intervals $I_i$ with multiplicities at most $t_i$, and do not use other parts. Our aim is to prove that the first condition is satisfied for $n = s_{i+1}-1$ and the second condition is satisfied for $n = t_{i+1}$.

Any partition of $s_{i+1}-1$ lying in $X$ may use only parts $\le t_i$ but, as the multiplicities are restricted, sum of all the parts $\le t_i$ equals
$$
    \sum_{l=1}^i t_l \sum_{j=s_l}^{t_l}j \leq t_i \sum_{j=1}^{t_i}j = \frac{t_i^2(t_i+1)}{2} < s_{i+1}-1\;.
$$
Hence easily $p(s_{i+1}-1, X) = 0$ for any positive integer $i$.

Let $K = (8 \sqrt{3})^{-1}$. To prove the second property we first show that for a fixed positive integer $i$ and any $n > \max \{t_i, n_0 \}$ we have
\begin{align} \label{ineq:help} 
    p(n)^{1-\sqrt{f(n)}} > n^{3s_i/2}\;.
\end{align}
Indeed, the definition of $f(n)$ implies $\sqrt{n}(1-\sqrt{f(n)}) = \log^{1+\varepsilon} n$ and $0 < f(n) < 1$, which, combined with the definition of $t_i$ and $n_0$, yields for any $n \geq \max \{t_i, n_0\}$ bound
\begin{align*}
    p(n)^{1-\sqrt{f(n)}} n^{-3s_i/2}
    &> K^{1-\sqrt{f(n)}} e^{C\sqrt{n}(1-\sqrt{f(n)})} n^{-1+\sqrt{f(n)}-3s_i/2} \\
    &> K e^{C\log^{1+\varepsilon} n} n^{-1-3s_i/2} \\
    &= K n^{C\log^\varepsilon n -1-3s_i/2}\ge Kn^4 > 1,
\end{align*}
by (\ref{eq:guessbound}) and the bounds $K>\frac{1}{16}$ and $n\ge2$. 
Now Theorem \ref{thm:main} and (\ref{ineq:help}) implies
\begin{align*}
    p(t_i, X) 
     \geq p_{- \left[ s_i \right] }(t_i) 
     > \frac12 p(t_i) \left( \frac{C}{2\sqrt{t_i}}\right)^{s_i} s_i! 
     > \frac12 p(t_i)^{\sqrt{f(t_i)}} \left( \frac{Ct_i}{2}\right)^{s_i} s_i!.
\end{align*}
We apply again asymptotics (\ref{eq:p(n)}) for $p(t_i)$ and $p(t_if(t_i))$ and have
\begin{align*}
    p(t_i)^{\sqrt{f(t_i)}}
     &> \left( \frac{e^{C\sqrt{t_i}}}{8\sqrt{3}t_i} \right)^{\sqrt{f(t_i)}}
     > \left( \frac{1}{8\sqrt{3}t_i} \right)^{\sqrt{f(t_i)}} 2t_if(t_i)\sqrt{3} p(t_if(t_i)) \\
     &> \frac18 (8t_i\sqrt{3})^{1-\sqrt{f(t_i)}} p(t_if(t_i))
     > \frac18 p(t_if(t_i))\;.
\end{align*}
Putting these results together we get that
$$
    p(t_i, X) > \frac1{16} p(t_if(t_i)) \left( \frac{Ct_i}{2}\right)^{s_i} s_i! > p(t_if(t_i)) 
$$
since $s_i\ge2,Ct_i\ge6$. That completes the proof.
\end{proof}

\section{Acknowledgments}
I would like to thank my supervisor Martin Klazar for useful conversations and corrections. The work was supported by the grant SVV–2017–260452 and SVV-2017-260456.

\bibliographystyle{plain}

\end{document}